\documentclass[12 pt]{amsart}

\newtheorem{theorem}{Theorem}[section]

\newtheorem{proposition}[theorem]{Proposition}

\newtheorem{definition}[theorem]{Definition}
\newtheorem{problem}[theorem]{Problem}
\theoremstyle{definition}
\newtheorem{remark}[theorem]{Remark}
\newtheorem{example}[theorem]{Example}
\numberwithin{equation}{section}
\newcommand{\CC}{\mathbb{C}}
\newcommand{\RR}{\mathbb{R}}
\newcommand{\BB}{\mathcal{B}}
\newcommand{\ZZ}{\mathbb{Z}}
\newcommand{\DD}{\mathbb{D}}
\renewcommand{\Re}{\mathfrak{Re}}
\renewcommand{\Im}{\mathfrak{Im}}
\newcommand{\HH}{\mathbf{H}}

\DeclareMathOperator{\spa}{span}
\DeclareMathOperator{\DAF}{\{\text{DAF}\}}

\title{Discrete analytic Schur functions}

\author[D.  Alpay]{Daniel Alpay}
\address{(DA)
Faculty of Mathematics, Physics, and Computation\\
Schmid College of Science and Technology\\
Chapman University\\
One University Drive
Orange, California 92866\\
USA}
\email{alpay@chapman.edu}

\author[D.  Volok]{Dan Volok}
\address{(DV) Mathematics Department\\ Kansas State University\\ 138 Cardwell Hall\\1228 N.  17th Street\\
  Manhattan, KS 66506, USA}
\email{danvolok@math.ksu.edu}

\begin{document}

\maketitle

\begin{abstract} We introduce the Schur class of functions,  discrete analytic on the integer lattice in the complex plane.  As a special case,  we derive the explicit form of discrete analytic Blaschke factors and solve the related basic interpolation problem.
\end{abstract}
\mbox{}\\


\section{Introduction}
The family of functions analytic and contractive in the open unit disk play an important role in various domains of theoretical and applied mathematics; they are in particular transfer functions of
linear dissipative systems. They bear various names, and in particular Schur functions and Schur analysis can be seen as a body of problems and methods pertaining to this family; see
\cite{MR2002b:47144} for background and references. Schur analysis has been generalized to various domains (for instance, upper triangular operators \cite{dede,add, MR2191194}, functions on
trees \cite{MR2182965,MR2481805},
quaternionic analysis \cite{zbMATH06658818}). In each case
motivation came in particular from linear system theory. The purpose of this work is to initiate a study of Schur analysis in the setting of discrete analytic functions.\\

Indeed, discrete complex analysis has drawn much attention in the recent years; see \cite{lovasz} and \cite{smirnov1}, and while many results of the classical complex analysis have been extended to
the discrete setting, the counterpart of function theory in the unit disk and Schur analysis seems to be lacking. The goal of the present paper is to fill this gap for functions that
are discrete analytic on the integer lattice in the complex plane. The latter is the simplest case of  lattice. Next steps would involve more general lattices and discrete Riemann surfaces.\\

Schur functions are contractive multipliers of the Hardy space of the disk $\mathbf H_2(\mathbb D)$. When considering the counterpart of Schur analysis in a new setting, the first task is
to define a natural counterpart of $\mathbf H_2(\mathbb D)$ (for instance, in the time-varying setting, one takes upper triangular Hilbert-Schmidt operators
from $\ell_2(\mathbb Z)$ into itself; Schur functions are then upper triangular contractions). By ``natural'' one means a Hilbert space together with a functional $E$ and
a co-isometric operator $V$ such that
\begin{equation}
I-E^*E=V^*V. 
\end{equation}

In the classical case, $E$ is point evaluation at the origin, and $V$ is the backward-shift operator. In the present setting, see Theorem \ref{bs1}.\\


The paper is organized as follows.  Section \ref{daf} reviews the classical definitions and results that go back to \cite{MR0013411} and \cite{MR0078441}.  Section \ref{dap} presents a basis of
discrete analytic polynomials suitable for "power series" expansions of discrete analytic functions.   In Section \ref{Hardy} this basis is used to define the
Hardy space of discrete analytic functions.  In Section \ref{Schur}  the Schur class of  contractive multipliers is identified by means of a positive kernel, and a Beurling-type result is
established.  Finally,  in Section \ref{Blaschke} a basic interpolation problem in the Hardy space of discrete analytic functions is solved by means of a suitable Blaschke product.

\section{Discrete analytic functions}\label{daf}
We will use the notation $\Lambda=\ZZ+i\ZZ$  for the lattice on which our functions are defined.
The following characterization of discrete analytic functions is due to J.  Ferrand; see \cite{MR0013411}.
\begin{definition}\label{defda}
The function $f:\Lambda\longrightarrow\CC$ is  a discrete analytic function 
if
$$
\dfrac{f(z+1+i)-f(z)}{1+i}=\dfrac{f(z+1)-f(z+i)}{1-i}, z\in\Lambda.$$
\end{definition}
A path in $\Lambda$ from $a$ to $b$ is a sequence $\gamma=(z_0,\ldots, z_n)$ of points in $\Lambda,$ such that $z_0=a,$ $z_n=b,$ and $$|z_{k}-z_{k-1}|=1,\quad 1\leq k\leq n.$$
The path $\gamma$ is closed if $a=b.$ The discrete integral of function $f:\Lambda\longrightarrow\CC$ over path $\gamma$ is defined by
\begin{equation}\label{di}\int_\gamma f\delta z=\sum_{k=1}^n \dfrac{f(z_{k-1})+f(z_k)}{2}(z_k-z_{k-1}).
\end{equation}
The main motivation of Definition \ref{defda} are the following two facts
(see \cite[pp. 340-341]{MR0078441}).
\begin{theorem}
\label{ol1}
\begin{enumerate} 
\item The function
$f:\Lambda\longrightarrow\CC$ is discrete analytic if,  and only if,  for every closed path $\gamma$ in $\Lambda$ it holds that $\int_\gamma f\delta z=0.$
\item If $f$ is discrete analytic, then $F(z)=\int_{0}^z f\delta z$ is also discrete analytic, where the integral is taken along any path from $0$ to $z$.
\end{enumerate}
\end{theorem}
 
From now on,  we shall use  notation $\DAF$ for the space of all discrete analytic functions on $\Lambda.$
 For our purposes,  it would be more convenient to re-state Definition \ref{defda} in terms of the difference operators
\begin{gather*}
\delta_xf (z):=f(z+1)-f(z),\\ \delta_yf(z):=f(z+i)-f(z).
\end{gather*}
Note that $\delta_x$ and $\delta_y$ commute.  The following is \cite[Theorem 4.1]{ajsv}.
\begin{theorem} \label{ol2} $\DAF=\ker\left(\alpha_-\delta_x+\alpha_+\delta_y+\dfrac{1}{2}\delta_x\delta_y\right),$ where
$$\alpha_\pm=\dfrac{1\pm i}{2}.$$
\end{theorem}

\section{A basis of discrete analytic polynomials}\label{dap}
One of the main differences between the classical complex analysis and the discrete setting is that the space $\{DAF\}$ is not closed under the pointwise multiplication - for example,  $z$ and $z^2$ are discrete analytic,  but $z^3$ is not. 
Our objective is to characterize the linear space of complex polynomials in $z$ and $z^*,$ which are discrete analytic on the lattice $\Lambda,$ and equip it with a suitable product.   
To this end,
 define operators\footnote{According to Theorem \ref{ol1},  if $f(z)$ is discrete analytic,  then so is $Zf(z).$}\footnote{Our use of notation $E_0^*$ is justified by  Proposition \ref{triv2} below, where $E_0^*$ is seen to be the adjoint of $E_0$ on an appropriate Hilbert space. } $Z:\DAF\longrightarrow\DAF,$ $E_0:\DAF\longrightarrow\CC,$ and $E_0^*:\CC\longrightarrow \DAF$ by
\[
 \begin{split} Zf(z)&=\dfrac{f(0)-f(z)}{2}+\int_0^zf\delta z;\\
E_0f&=f(0),\\ (E_0^*k)(z)&=k, \quad k\in\mathbb C.
\end{split}
\]

\begin{theorem}\label{bs1}
  It holds that
$$ \delta_x Z=I,\quad Z\delta_x=I-E_0^*E_0.$$
\end{theorem}
\begin{proof}
Let $f\in\DAF.$ By definition of the integral \eqref{di} and Theorem \ref{ol1},
$$(\delta_x Zf)(z)=\dfrac{f(z)-f(z+1)}{2}+\int_{z}^{z+1} f\delta z=f(z),$$
which proves the first identity.
Furthermore,  since
$$\delta_x(Z\delta_x-I)f=0,$$
Theorem \ref{ol2} implies that 
$$\delta_y(Z\delta_x-I)f=0,$$
and hence $(Z\delta_x-I)f$ is a constant function;
$$((Z\delta_x-I)f)(z)=(Z\delta_xf)(0)-f(0)=-f(0).$$
\end{proof}

For $n\in\ZZ_+$ we define $z^{(n)}=(Z^n1)(z).$ Then $z^{(n)}$ is a discrete analytic polynomial of degree $n.$ 
Since  (see \cite[p. 353]{MR0078441}) the linear space of discrete analytic polynomials of degree less or
equal to $n$ is of dimension $n+1,$
$\{z^{(n)}:n\in\ZZ_+\}$ is a basis of discrete analytic polynomials.

We define the {\em convolution} product $\odot$ of  discrete analytic polynomials as
$$\left(\sum_{m=0}^Ma_mz^{(m)}\right)\odot\left(\sum_{n=0}^Nb_nz^{(n)}\right)=\sum_{m=0}^M\sum_{n=0}^Na_mb_nz^{(m+n)}.$$
The convolution product can be extended to a suitable class of discrete analytic functions; see Section \ref{Schur} below.

\begin{remark} One can find in the literature several different bases of discrete analytic polynomials -- see,  for instance,  
\cite{zeil}.  The basis $\{z^{(n)}\},$ presented above,  was previously introduced (in a different way) by the authors in \cite{ajsv}. 
\end{remark}

\section{Hardy space of discrete analytic functions}\label{Hardy}
In order to get a more explicit description of the polynomials $z^{(n)},$ we construct a generating function. 
The following is \cite[Proposition 5.1, p. 406]{ajsv}. We give a proof for completeness.

\begin{proposition} \label{genpr}For $t\in\DD$ and $z\in\Lambda$ define
$$G_z(t)=(1+t)^{\Re(z)}\left(\dfrac{1+\alpha_+t}{1+\alpha_-t}\right)^{\Im(z)},$$
where $\alpha_\pm$ are as in Theorem \ref{ol2}.
Then
$$G_z(t)=\sum_{n=0}^\infty t^n z^{(n)}.$$
\end{proposition}
\begin{proof}
Write $G_z(t)=\sum_{n=0}^\infty t^n g_n(z).$ Since
$$G_{z+1}(t)-G_z(t)=G_z(t)t\text{ and }G_{z+i}(t)-G_z(t)=G_z(t)\dfrac{it}{1+\alpha_-t},$$
Theorem \ref{ol2} implies that,  for every $t\in\DD,$ $G_z(t)$ is discrete analytic as a function of $z.$ Hence $g_n(z)$ are discrete analytic functions, as well, and
$$\delta_xg_n(z)=g_{n-1}(z),\quad n=1,2,\dots$$
Since $g_0(z)=1$ and 
$g_1(0)=g_2(0)=\dots=0,$ Theorem \ref{bs1} implies that
$g_n(z)=z^{(n)}.$
\end{proof}
We denote $\Lambda_+=\{z\in\Lambda\,:\,\Re(z)\geq 0\}.$
\begin{theorem} \label{cor} If $\hat f(n)$, $n=0,1,\ldots$, is a sequence in $\mathbb C$,
such that
$$\limsup_{n\rightarrow\infty}\sqrt[n]{|\hat f(n)|}<\sqrt{2},$$
then the series
$$f(z)=\sum_{n=0}^\infty \hat f(n)z^{(n)}$$
converges absolutely for $z\in\Lambda_+.$  Moreover,  $f(z)$ is discrete analytic in $\Lambda_+$ and
$$\label{unco}\hat f(n)=(\delta_x^nf)(0),\quad n=0,1,\ldots$$
\end{theorem}
\begin{proof}
Note that,  for $z\in\Lambda_+,$ the function $G_z(t)$ of Proposition \ref{genpr} is analytic with respect to $t$ in the disk $|t|<\sqrt{2}.$ Thus
$$\limsup_{n\rightarrow\infty}\sqrt[n]{|z^{(n)}|}=\dfrac{1}{\sqrt{2}},\quad z\in\Lambda_+,$$
and the absolute convergence follows.  The rest follows from the observation that, in view of Theorem \ref{bs1},
$$\delta_x z^{(n+1)}=z^{(n)},\quad n=0,1,\ldots.$$
\end{proof}
Next we introduce an analogue of the Hardy space in the present setting.
\begin{definition}
The Hardy space $\HH_2(\Lambda_+)$ is the Hilbert space of discrete analytic functions
$$f(z)=\sum_{n=0}^\infty \hat f(n)z^{(n)},\text{ where } \sum_{n=0}^\infty |\hat f(n)|^2<\infty,$$
equipped with the inner product
$$\langle f,g\rangle=\sum_{n=0}^\infty \hat f(n)\hat g(n)^*.$$
\end{definition}

\begin{proposition}\label{triv2}\begin{enumerate}
\item
The space $\HH_2(\Lambda_+)$ is $Z$-invariant and $\delta_x$-invariant.  On this space $\delta_x$ is the adjoint of $Z,$ and $E_0$ is the adjoint of $E_0^*.$
\item The space $\HH_2(\Lambda_+)$ is a reproducing kernel Hilbert space with the reproducing kernel
$$K_w(z)=\sum_{n=0}^\infty z^{(n)}w^{(n)*}.$$
\item 
The linear operator
$$T: \HH_2(\Lambda_+)\longrightarrow \HH_2(\DD),\quad Tz^{(n)}=z^n,$$
is unitary.  
\item 
$$
TK_w(z)=G_{w^*}(z).
$$
\end{enumerate}\end{proposition}
\begin{proof}
Parts 1 -- 3 are straightforward.  To prove part 4,  it suffices to show that
$$z^{(n)*}=(z^*)^{(n)},\quad z\in\Lambda.$$
Observe that,  as follows from Proposition \ref{genpr},  
$$\Im(z^{(n)})=0,\quad z\in\Lambda\cap\RR,$$
and that $(z^*)^{(n)*}$ is a discrete analytic polynomial, which coincides with $z^{(n)}$ on the real axis.  In view of Theorem \ref{cor}, 
$$(z^*)^{(n)*}=z^{(n)},\quad z\in\Lambda.$$

\end{proof}
\section{Schur class.}\label{Schur}
\begin{definition}
Let $s(z)$ be discrete analytic on $\Lambda_+.$ Then, $s(z)$ is a discrete analytic Schur function if 
$$\forall z\in\Lambda_+,\quad \sum_{n=0}^\infty |(Z^ns)(z)|^2<\infty,$$
and the kernel
$$K^s_w(z)=\sum_{n=0}^\infty z^{(n)}w^{(n)*}-(Z^ns)(z)(Z^ns)(w)^*$$ is positive in $\Lambda_+.$
\end{definition}

\begin{theorem} Let $s(z)$ be a discrete analytic Schur function.  Then $s(z)$ belongs to the Hardy space $\HH_2(\Lambda_+),$  and for every $f(z)$ from $\HH_2(\Lambda_+)$ the convolution product
$$(s\odot f)(z)=\sum_{n=0}^\infty\sum_{m=0}^n \hat s(m)\hat f(n-m)z^{(n)}$$
converges absolutely in $\Lambda_+$ to a discrete analytic function.  Moreover,
 the linear operator
$$M_sf=s\odot f$$
is a contraction on  $\HH_2(\Lambda_+).$

Conversely,  if $s\in\HH_2(\Lambda_+)$ is such that the multiplication operator $M_s$ is a contraction 
on  $\HH_2(\Lambda_+),$ then $s(z)$ is a discrete analytic Schur function.
\end{theorem}

\begin{proof}
The positivity of $K^s_w(z)$ implies the existence of a contraction $C$ on $\HH_2(\Lambda_+),$
such that
$$(CK_w)(z)=\sum_{n=0}^\infty z^{(n)}(Z^n s)(w)^*.$$
The adjoint of $C$ is the multiplication operator $M_s.$ In particular, $s(z)=(M_s1)(z)$ is in $\HH_2(\Lambda_+).$

The converse statement follows from the identity
$$K^s_w=(I-M_sM_s^*)K_w.$$
\end{proof}
\begin{remark} Let $s\in\HH_2(\Lambda_+).$ Then $s(z)$ is a discrete analytic Schur function if,  and only if,
$(Ts)(z)$ is a Schur function in the open unit disk.  In this case
$$(T(s\odot f))(z)=(Ts)(z)(Tf)(z),\quad f\in\HH_2(\Lambda_+),z\in\DD.$$
\end{remark}

\begin{example} Let $\alpha\in\DD.$ Consider
$$s(z)=(1-|\alpha|)G_z(\alpha),\quad z\in\Lambda_+.$$
Then
$$(Ts)(z)=\dfrac{1-|\alpha|}{1-\alpha z},\quad z\in\DD,$$
hence $s(z)$ is a discrete analytic Schur function.
Since $$s(z)\odot (1-\alpha z)=1-|\alpha|,$$ 
we denote
$$G_z(\alpha)=(1-\alpha z)^{-\odot},\quad \alpha\in\DD, z\in\Lambda.$$
\end{example}

\begin{theorem}[A Beurling-type theorem] \label{beurling} Let $\HH$ be a closed subspace of $\HH_2(\Lambda_+).$ Then 
$\HH$ is $Z$-invariant if,  and only if,  there exists a discrete analytic Schur function $s(z),$ such that $M_s$ is an isometry with the range $\HH.$  
\end{theorem}
\begin{proof}
$\HH$ is $Z$-invariant if,  and only if,  $T\HH$ is a subspace of $\HH_2(\DD),$ invariant under multiplication by $z$. In view of the classical Beurling theorem \cite{beurling},  this is the case if,  and only if,  $T\HH=s_0\HH_2(\DD),$ where $s_0(z)$ is inner.  The latter condition is equivalent to $\HH=s\odot\HH_2(\Lambda_+),$ where $s(z)=(T^*s_0)(z)$ is a discrete analytic Schur function,  and $M_s$ is an isometry.
\end{proof}

\section{Blaschke factors}\label{Blaschke}
Let $\lambda\in\Lambda_+.$ SInce the one-dimensional span of $K_\lambda$ in $\HH_2(\Lambda_+)$ is not,  in general,  $\delta_x$-invariant,  we consider instead a homogeneous interpolation problem in the rectangle
$$R_\lambda=\{z\in\Lambda_+:\Re(z)\leq\Re(\lambda),|\Im(z)|\leq|\Im(\lambda)|,\Im(z)\Im(\lambda)\geq 0.\}$$

 \begin{problem} \label{prob}Describe all $f\in\HH_2\Lambda_+)$ such that 
$$f(z)=0,\quad z\in R_\lambda.$$
\end{problem}
Denote $\HH_\lambda=\spa\{K_w:w\in R_\lambda\}.$ Then $\HH_\lambda$ is a finite-dimensional $\delta_x$-invariant subspace of $\HH(K).$ The set of solutions of Problem \ref{prob} forms the orthogonal complement in $\HH_2(\Lambda_+)$ of $\HH_\lambda$ and is,  in view of Theorem \ref{beurling},  of the form
$\BB_\lambda\HH_2(\lambda_+),$ where $\BB_\lambda(z)$ is a discrete analytic Schur function,  such that the convolution product with $\BB_\lambda(z)$ is an isometry on $\HH_2(\Lambda_+)$.

\begin{definition} The discrete analytic Schur functions
$$
\BB_\pm(z)=\alpha_\mp+\alpha_\pm z\odot (1+\alpha_\pm z)^{-\odot}\alpha_\mp,
$$
where $\alpha_\pm$ are as in Proposition \ref{genpr},
are called Blaschke factors in,  respectively,  the upper and lower quadrants of $\Lambda_+.$
\end{definition}

\begin{theorem} \label{blat} Up to an unimodular  multiplicative constant,
$$\BB_\lambda =z^{(\Re(\lambda)+1)}\odot
\BB_\pm^{\odot|\Im(\lambda)|},$$
where the sign "$\pm$" stands for "$+$" if $\Im(\lambda)\geq 0$ and "$-$" otherwise. 
\end{theorem}
\begin{proof}
Note that $T\HH_\lambda$ is a finite-dimensional backward-shift invariant subspace of $\HH_2(\DD).$ Thus
$T\BB_\lambda$ is a Blaschke product of degree equal to
 $$\dim(T\HH_\lambda)=
\Re(\lambda)+|\Im(\lambda)|+1.$$
Furthermore,  according to Proposition \ref{triv2},
$$T\HH_\lambda=\spa\{G_{w^*}\,:\,w\in I_\lambda\}.$$ 
Suppose that $\Im(\lambda)\geq0.$ Then for $w\in I_\lambda$ 
$$\dfrac{G_{w^*}(z)}{z^{\Re(\lambda)+1}(T\BB_+)(z)^{\Im(\lambda)}}=
\dfrac{(1+z)^{\Re(w)}(1+\alpha_-z)^{\Im(w)}(1+\alpha_+z)^{\Im(\lambda-w)}}{z^{\Re(\lambda)+1}(z+\alpha_-)^{\Im(\lambda)}}$$
is analytic outside the unit circle and vanishes at $\infty.$ Thus
$$z^{\Re(\lambda)+1}(T\BB_+)(z)^{\Im(\lambda)}$$ is a Blaschke product of degree $\Re(\lambda)+\Im(\lambda)+1,$
which is orthogonal to $T\HH_\lambda$ and,  therefore,  coincides with $T\BB_\lambda$ up to an unimodular constant.  The case $\Im(\lambda)<0$ is similar.
\end{proof}
\begin{remark} It was proved in \cite[p. 352]{MR0078441} that all non-trivial polynomial solutions to Problem \ref{prob} are of degree at least  $\Re(\lambda)+|\Im(\lambda)|+1.$  Theorem \ref{blat} allows to give an explicit description: a discrete analytic polynomial $p(z)$ vanishes on the rectangle $R_\lambda$ if, and only if, it is of the form
$$p(z)=z^{(\Re(\lambda)+1)}\odot(z+\alpha_\mp)^{\odot|\Im(\lambda)|}\odot q(z),$$
where $q(z)$ is an arbitrary discrete analytic polynomial, and the choice of sign is the same as in the theorem.
In other words,  in the ring of discrete analytic polynomials equipped with the  product $\odot,$ the polynomial solutions to Problem 6.1 form a principal ideal.  
\end{remark}

\section*{Acknowledgments}
Daniel Alpay thanks the Foster G. and Mary McGaw Professorship in Mathematical Sciences, which supported this research.

\end{document}